\documentclass{article}
\newcommand{\Mult}{{\mathcal M}} 

\newcommand{\compose}{\circ}
\renewcommand{\H}{\ensuremath{{\mathcal H}}} 
\newcommand{\C}{\mathbb C}
\newcommand{\N}{\mathbb N}

\newcommand{\M}[1]{M_{#1}(\C)} 
\newcommand{\BH}{\ensuremath{B(\H)}} 
\newcommand{\tensor}{\otimes} 
\newcommand{\mintensor}{\otimes_{min}} 
\newcommand{\HtH}{\ensuremath{\H\tensor\H}}
\newcommand{\BHH}{\ensuremath{B(\H\tensor\H)}} 

\newcommand{\isom}{\cong}
\newcommand{\Ah}{\ensuremath{\widehat{A}} } 
\newcommand{\Bh}{\ensuremath{\widehat{B}} } 

\newcommand{\fancyK}{\ensuremath{\widetilde{K}}} 
\renewcommand{\fancyK}{\ensuremath{{K}}} 
\newcommand{\antipode}{\ensuremath{\kappa}} 
\renewcommand{\antipode}{\ensuremath{S}} 
\newcommand{\antipodeh}{\ensuremath{\widehat\kappa}} 
\renewcommand{\antipodeh}{\ensuremath{\widehat{S}}} 

\newcommand{\counit}{\epsilon} 

\newcommand{\haar}{\tau}
\newcommand{\haarh}{\widehat{\varphi}}
\renewcommand{\haar}{\varphi}
\renewcommand{\haarh}{{\varphi}}

\newcommand{\flip}{\sigma}
\newcommand{\F}{\mbox{\ensuremath{\mathcal F}}} 
\newcommand{\inv}{\ensuremath{{}^{-1}}}
\newcommand{\Id}{\mbox{\rm Id}}

\newcommand{\arrow}{\rightarrow}

\newcommand{\coproduct}[1][\empty]{\ensuremath{\delta_#1}}
\renewcommand{\coproduct}[1][\empty]{\ensuremath{\Delta_#1}}

\newcommand{\convolution}{\diamond}   \newcommand{\convolve}{\convolution}
\newcommand{\boxproduct}{\,\square\,} 
\newcommand{\comment}[1]{\-\marginpar[\raggedleft\footnotesize\it\textcolor{Sienna4}{#1\smallskip}]{\raggedright\footnotesize\it\textcolor{Sienna4}{#1\smallskip}}}
\renewcommand{\comment}[1]{}
\newcommand{\compact}{\ensuremath{\mathcal K}}
\newcommand{\adjointable}{\ensuremath{\mathcal L}}
\newcommand{\ip}[2]{\ensuremath{\left\langle #1,#2\right\rangle}}
\newcommand{\Cu}{\ensuremath{{\mathcal C}\hspace{-.8pt}u}}
\renewcommand{\fancyK}{\ensuremath{{K}}} 

\newcommand{\Add}{\ensuremath{A^{**}}}
\newcommand{\Bdd}{\ensuremath{B^{**}}}
\newcommand{\leqCu}{\preceq_{\scriptscriptstyle {\mathcal C}\hspace{-.75pt}u}}
\newcommand{\geqCu}{\succeq_{\scriptscriptstyle {\mathcal C}\hspace{-.75pt}u}}
\newcommand{\blackadar}{\sim_{s}} 
\newcommand{\HA}{{\mathcal H_A}}
\usepackage{amsmath}%
\usepackage{fullpage}
\usepackage{amsfonts}%
\usepackage{amsthm}
\usepackage[hug,balance,dpi=1200]{diagrams}

 \newtheorem{Theorem}{Theorem}
 \newtheorem{Lemma}{Lemma}
 \newtheorem{Corollary}{Corollary}
 \newtheorem{Proposition}{Proposition}

 \newtheorem{Example}{Example}
 
 \newtheorem{Remark}{Remark}
 \newtheorem{Definition}{Definition}

\begin{document}

\centerline{\Large Cuntz Semigroups of Compact-Type Hopf~C*-Algebras}
\centerline{\large Dan  Ku\v{c}erovsk\'{y} 
} 


\let\thefootnote\relax
\footnote{Address: Department of Mathematics and Statistics, University of New Brunswick, Fredericton NB E3B 5A3, Canada  \qquad dkucerov@unb.ca}

\textbf{Abstract:}{ The classical Cuntz semigroup has an important role in the study of C*-algebras, being one of the main invariants used to classify recalcitrant C*-algebras up to isomorphism. We consider C*-algebras that have Hopf algebra structure, and find additional structure in their Cuntz semigroups.
We show how to use results from the C*-algebraic classification program to show that in  many cases, isomorphisms of Cuntz semigroups that respect this additional structure can be lifted to Hopf algebra (bi)isomorphisms, up to a possible flip of the co-product. This shows that the Cuntz semigroup provides an interesting invariant of C*-algebraic quantum groups.
}

\footnote{
This paper was partially supported by NSERC. This paper was partially supported by the grant H2020-MSCA-RISE-2015-691246-QUANTUM DYNAMICS and 3542/H2020/2016/2.}

\section{Introduction}

 In this paper, we address two problems, the second depending on the first. The first problem arises as follows:   in the algebraic setting, there is a product operation on modules over a Hopf algebra, as is discussed further at the end of Section \ref{Section2}, or in e.g., \cite{kuce.classify.Hopf.by.Ktheory0}. However, in the C*-algebraic quantum group setting, there are  reasons, coming from the classification program for C*-algebras, motivating us to seek a Hilbert module version of this construction. It is desirable to allow some equivalence relations on the set of  Hilbert modules, and naturally the equivalence relation should be compatible with the product operation.  Drawing a parallel with the C*-algebraic Cuntz semigroup motivates the use of a~coarser equivalence relation on Hilbert modules than the isomorphism relation that we first consider. These subtleties, as well as a desire to take advantage of known results about Cuntz semigroups, lead us to study the implied equivalence relation(s) on generators of Hilbert modules.   We obtain an invariant for C*-algebraic quantum groups,  generalizing the Cuntz semigroup for C*-algebras, that we term the \textit{Cuntz semiring}. For C*-algebras, it is often true that if the algebras belong to the same class in a sense that we discuss later, then isomorphism at the level of Cuntz semigroups implies isomorphism at the level of C*-algebras. The second problem we study is to decide when a similar statement holds at the level of Cuntz semirings and C*-algebraic quantum groups. Ultimately, and perhaps inevitably, we will have to restrict our study to the Kac case in order to obtain results on the second problem, namely obtaining  Hopf algebra isomorphism from Cuntz semigroup isomorphisms. This is because Cuntz semirings seem to have good stability properties under deformation, but in general, without restrictive conditions, C*-algebraic quantum groups do not. 

The Elliott program \cite{classification.overview}  is to classify C*-algebras using K-theoretical invariants. Recently, in~this program, a popular classifying functor on appropriate classes of C*-algebras has been the Cuntz semigroup.
 The partial solution to the  second of the above problems thus fits into the framework of the Elliott program, generalized to the setting of C*-algebraic quantum groups, see \cite{kuc.autos,kuce.classify.Hopf.by.Ktheory0,kuce.classify.Hopf.by.Ktheory} for related results.
Our main result is  probably Theorem \ref{th:banacheweski.Hopf}  and its corollaries, such as Corollary  \ref{cor:Cuntz.banacheweski.Hopf} and Theorem \ref{th:banacheweski.Hopf.ASH}.

  There are two  Cuntz semigroups that are commonly used, denoted here by $\Cu(A)$ and $W(A),$ respectively. The first has better functorial properties; being, for example, continuous with respect to inductive limits, but the second was historically prior. The definition of the Cuntz semigroup, $\Cu(A),$ of a C*-algebra, $A,$ can be given, in the Hilbert module picture, as follows: first, consider the set of countably generated Hilbert modules over the algebra, endowed with direct sum as the addition operation. There exists a certain order relation on such modules, related to, but more complicated than inclusion,  although reducing to inclusion in some cases. Antisymmetrizing the order relation defines an equivalence relation on Hilbert modules, and the quotient of the set of countably generated Hilbert modules by this equivalence relation is the Hilbert module picture of the Cuntz semigroup, $\Cu(A).$ The semigroup $W(A)$ is the subsemigroup of $\Cu(A)$ consisting of elements contained in  $A^n.$ See  \cite{CEI}. 

The Cuntz semigroup has become a standard technical tool for studying C*-algebras. We consider the case of C*-algebras with Hopf algebra structure, and, in Section \ref{Section2}, after developing some preliminaries, we
 define  a convolution product on the Cuntz semigroup of a C*-algebraic quantum group, in the case of stable rank one, and establish  a few of its properties. We find that the product has good behaviour in general, and that products with a fixed element are Cuntz semigroup morphisms.  We generalize to higher stable rank in Section \ref{Section3}, where we furthermore relate our results to the operator picture of the Cuntz semigroup. Our main result at this point is that there is a (semi)ring structure on the Cuntz semigroup of a compact-type C*-algebraic quantum group: the main issue having been compatibility with the equivalence relations that define the Cuntz semigroup.

We now give an outline of what is covered in the various sections of the paper.
In Section \ref{Section2}, we obtain a kind of Mazur--Ulam lemma, extended to the case of Hilbert modules. The~classic Mazur--Ulam theorem, see \cite{MazurUlam}, states that a surjective isometry of normed spaces is necessarily affine (or linear, depending on the exact form of the theorem).  Our lemma extends the theorem slightly to the case of Hilbert modules, in such a way that the additional structure of a Hilbert module is taken into account. Our result gives a condition for a mapping of Hilbert modules that is an isometry in a certain weak sense to be in fact a unitary operator at the Hilbert module level. This~technique is used  to show that the product we define respects the complicated and delicate equivalence relations of the Cuntz semigroup. There is some relationship between the problem of showing that equivalence classes are respected and the problem of finding Hilbert module connections~\cite{CS}, and this is discussed briefly at the end of Section \ref{Section2}. In this section, there is also an example of the product, namely, the finite-dimensional case.
It seems to be useful, and perhaps necessary, to~have a product defined by a form of convolution, since this allows the Fourier transform---see \cite{kahng} for information about the Fourier transform---to~be brought to bear in proving  Elliott-type results. We~study the relationship of the product with convolution at the start of Section \ref{Section4} and then develop the double dual picture of the Cuntz semigroup: this will be technically useful and is perhaps a timely contribution since double dual techniques are currently fashionable in the classification program for C*-algebras.
The main result of Section \ref{Section4} is that we obtain, in embryonic form, isomorphism results for C*-algebraic quantum groups. The overall main results are in Section \ref{Section5}, which also briefly considers the case of K-theory. The significance of these isomorphism results lies, first of all, in the implication that a relatively small object (the Cuntz semiring) completely characterizes a quantum group (at~least~in a given class); and secondly,   in parallel with the C*-algebraic Elliott classification program. Most~of the results of this classification program are obtained from theorems that show that within certain classes of C*-algebra, maps at the level of K-theory or the Cuntz semigroup can be lifted to C*-algebraic maps. It is quite often the case in the C*-algebraic setting that the K-theory groups must be regarded as being augumented by some additional information. The results that we obtain for C*-algebraic quantum groups are of precisely this sort.
We have previously considered classifying C*-algebraic quantum groups by K-theory  \cite{kuce.classify.Hopf.by.Ktheory,kuce.classify.Hopf.by.Ktheory0}, but in those results, we considered the K-theory of the discrete dual, and made extensive use of discreteness, or even finiteness, in the proofs. In general, it seems most natural to classify using an invariant that lives on the compact-type side of the given C*-algebraic quantum group, thus we focus on the Cuntz semigroup and K-theory group of a compact-type C*-algebraic quantum group.  The Cuntz semigroup is generally thought, in the context of the classification program for C*-algebras, to be the appropriate classifying functor for nonsimple C*-algebras. A  C*-algebraic quantum group is nonsimple as a C*-algebra, therefore, a classification based on the Cuntz semigroup should ultimately have greater scope, in theory, than a classification based on K-theory. Nevertheless, we do briefly address the case of K-theory in the last section.

A Hopf algebra is a bi-algebra with an antipode map $\antipode.$ See  \cite{abe,KP1966} for information on Hopf algebras.
Compact quantum groups were defined first by Woronowicz \cite{Woronowicz1998}. Multiplicative unitaries were introduced by Baaj and Skandalis in \cite{BS} (Chapter 4). 
A compact-type C*-algebraic quantum group is unital as an algebra, and has structure maps that are compatible with the C*-algebraic structure. Each compact-type C*-algebraic quantum group carries with it an algebraic quantum group as a~dense subset, and an enveloping Hopf--von Neumann algebra. See  \cite{timmermann} for a discussion.
The algebraic elements of a compact-type C*-algebraic quantum group $A$ will be denoted $A_0$ and the enveloping Hopf--von Neumann algebra by $\Add.$ The dual object of a compact-type C*-algebraic quantum group is both a discrete multiplier Hopf algebra and a C*-algebra, $B.$ The dual is called a discrete-type C*-algebraic quantum group.  The condition of stable rank 1 is sometimes used below: being separable and of stable rank one  means that there exists, at the level of normed algebras, a countable dense set of invertible elements. We assume the usual density properties, sometimes called cancellation properties,
$A\tensor A=\overline{(A\tensor 1)\coproduct{} (A)}$ and $A\tensor A=\overline{(1\tensor A)\coproduct{} (A)}.$ We generally assume nuclearity, which is a mild technical condition, equivalent to co-amenability, that makes states on the Cuntz semigroup easier to handle and also improves the technical behavior of tensor products. See \cite{tomatsu} for the equivalence of co-amenability and nuclearity, also observed by \'E. Blanchard. Nuclearity allows constructing faithful Haar states, and  we  assume that our C*-algebraic quantum groups have faithful Haar states. Our notation  denotes co-products by $\coproduct{},$ antipodes by $\antipode,$ pairings by $\beta(\cdot\,,\cdot),$ Haar states by $\haar,$ and  co-units by $\counit.$
We denote the flip, on~a~tensor product, by $\flip$.

\section{The Gelfand-Naimark-Segal construction, and some useful
	 Modules and Maps}\label{Section2}

There are two main techniques used in this section. The first is to observe that the  inner product module $A\tensor \H,$ where $\H$ is a suitably chosen Hilbert space and $A$ is a compact-type C*-algebraic quantum group, can be embedded in a natural way into $\H\tensor \H,$ and that furthermore a multiplicative unitary can be chosen in such a way that it has certain algebraic properties when restricted to the copy of $A\tensor \H$ inside $\H\tensor \H.$ This amounts to a linearized version of well-known properties of multiplicative unitaries. The second is an automatic continuity result for Hilbert space mappings that have suitable algebraic properties when restricted to a Hilbert module embedded inside a given Hilbert space. By~the term automatic continuity, we mean that a mapping that is continuous with respect to one topology may be shown to be continuous in a stronger topology in the presence of suitable algebraic properties.

We briefly recall some relevant facts about Hilbert modules from \cite{lance}, although our notation is as in \cite{CEI}, which is the foundational paper on Cuntz semigroups in the Hilbert module picture. An~\textit{inner product module over a C*-algebra $A$} is a right $A$-module $E$ equipped with an $A$-valued mapping $E\times E\arrow A$ that is usually denoted $\ip{\cdot}{\cdot}_A$. This map, usually called the \textit{inner product}, is  $A$-sesquilinear, meaning $A$-linear in the second entry, following \cite{lance} (p. 2). It must also be non-degenerate, in~the sense that $\ip{x}{x}$ is zero if and only if $x$ is zero, and positive-definite.  Note that we are not allowing the inner product to be semi-definite. There is a natural norm obtained from the C*-norm of $A$, namely $\|x\|_E:=\|\!\ip{x}{x}\!\|^{1/2}.$  If the module is complete with respect to the norm, we have a (right) Hilbert $A$-module (and it is straightforward to complete an inner product module).

\begin{Definition} Let $A_1$ and $A_2$ be  unital C*-algebras with  faithful states, denoted $f_i.$ Let $V\colon E_1\arrow E_2$ be a $\C$-linear map of countably generated inner product (Hilbert) $A$-modules $E_1$ and $E_2.$ We say that $V$ is a~\emph{2-isometric map} with respect to $f_1$ and $f_2$ if
$ f_1\left( \ip{x}{y}_{E_1}\right)= f_2\left(\ip{Vx}{Vy}_{E_2}\right).$
\end{Definition}
As defined above, the 2-isometric property is really just a new term for the familiar property of extending to an unitary or isometry on an enveloping Hilbert space: it seems better to reserve the term unitary for Hilbert module unitaries, which are $A$-linear. A 2-isometric map need not be an~$A$-module~map.

 Multiplicative unitaries were introduced by Baaj and Skandalis in  \cite{BS} (Chapter 4), see also \cite{BBS}: given a compact quantum group  $A$ with coproduct $\coproduct{}\colon A\arrow  A \mintensor A,$ and Haar state $\haar\colon A\arrow\C$ which is faithful on $A$  (i.e., $A$ is the reduced form of the compact quantum group), let $\H$ be the Hilbert space $L^2 (A,\haar),$  let $e\in \H$  be $e=\Lambda_{\haar}1.$ Then, $Ae$ is dense in $\H$ and one defines a multiplicative unitary $V\in\adjointable(\HtH)$ by $Vae\tensor be=\coproduct{}(a)e\tensor be.$ Another multiplicative unitary  $W,$ satisfying $W^* ae\tensor be=\coproduct{}(b)ae\tensor e,$ was~introduced in the locally compact case by Kustermans and Vaes \cite{KustermansVaes}, see also the notes by Maes and van Daele  \cite{MaesVanDaele}.
The $\C$-linear map $V\colon A\tensor A \arrow A\tensor\H$ that will be constructed in Lemma~\ref{lem:special.module} below could in fact be  defined as: $x\mapsto W^{*} x(1\tensor e),$ where $W$ is the above multiplicative unitary, which~is in $\Mult(A\tensor\compact(\H)).$ (We should mention that although $A\tensor\H$ and $A\tensor A$ are right Hilbert $A$-modules, this map $V$  is not intended to be $A$-linear.)

Thus, Lemma \ref{lem:special.module} can be eventually deduced from the locally compact case, for example, starting from \cite{KustermansVaes} (Proposition 3.17). We do, however, present a self-contained C*-algebraic proof below, for the reader's convenience.

As a C*-algebraic preliminary, we remark that if $A$ is a compact-type  C*-algebraic quantum group with a faithful (Haar) state, assumed nuclear at the C*-algebraic level, we may construct the GNS Hilbert space $\H$ associated with it, sometimes called a rigged Hilbert space. At the level of normed vector spaces, this construction embeds  $A\tensor A$ in $A\tensor\H,$ which is itself embedded in $\H\tensor\H.$   The~tensor product $A\tensor\H$ has a natural $A$-module and inner product module structure; an  $A$-valued inner product is constructed  by taking the (exterior) tensor product of  inner products on $A$ and $\H$. The~fact that $A$ and $\C$ commute simplifies the proof of positive-definiteness; see also \cite{lance} (pp. 34--35).

\begin{Lemma}Let $A$ be a separable, nuclear, and compact-type  C*-algebraic quantum group, with faithful right Haar state $f$.  There exists a $\C$-linear map  $V\colon {  A\tensor A}\arrow A\otimes\H$ such that
$V(\coproduct{}(a)x)=aV(x),$ for all $a\in A$ and $x\in A\tensor A.$  This map is a 2-isometry with respect to the faithful states $f\tensor f$ on $A\tensor A$ as a Hilbert module over itself, and $ f $ on the inner product $A$-module $A\tensor\H.$
\label{lem:special.module}\end{Lemma}
\begin{proof}  Let us first define the inverse map $T\colon A\tensor \H\arrow A\tensor A.$
The domain space for the map is  $A\tensor\H,$ a Hilbert space with coefficients in $A,$ regarded as infinite sequences $(a_i)$ of elements of $A.$
Since $A$ has a faithful Haar state, we may embed $A$ in a Hilbert space by the GNS construction. Taking, then, a countable dense subset consisting of algebraic elements of $A$ and applying the Gram--Schmidt process at the Hilbert space level, we obtain a countable basis $(g_i)$ for the Hilbert space. Since these elements $g_i$ are still (algebraic) elements of $A$, we then define a possibly unbounded map $T$ by $T\colon (a_i)\mapsto \sum \coproduct{A}(a_i)(g_i \tensor 1),$ where the sum on the right is required to converge in norm.  It is clear that the domain of this map is an $A$-module under the diagonal action of $A$ and that $T(ax)=\coproduct{A} (a)T(x).$ The range is norm-dense in $A\tensor A$ because of the density property $A\tensor A=\overline{\coproduct{} (A)(A\tensor 1)}.$

 We now show $T$ is continuous in an appropriate sense, by showing that it has  the 2-isometric property. Let the sequence $(a_i)$ denote an element in $A\tensor\H.$ We obtain,
  by the properties of the Haar state and the fact that the coproduct is a *-homomorphism:
 \begin{equation} \begin{split}(\Id\tensor f)\ip{T((a_i))}{T((a_j))}&= (\Id\tensor f)\left(\sum_{ij} (g_i^*\tensor1)\coproduct{}(a_i^*a_j)(g_j\tensor1)\right)\\ &=\sum_{ij} f(a_i^*a_j)g_{i}^* g_j.
\end{split}\label{eqn:basis.decomp}\end{equation}

Applying $f$ one more time, to both sides, we have:
\begin{equation*}(f\tensor f)\ip{T((a_i))}{T((a_j))}=\sum_{ij} f(a_i^*a_j)f(g_{i}^* g_j).\end{equation*}

It follows from the basis property of the $g_i$ that the right hand side simplifies to $\sum_{i} f(a_i^*a_i)$ which is evidently equal to $f\left(\ip{(a_i)}{(a_i)}_{A\tensor\H}\right).$ This establishes the 2-isometric property. 
Denote the inverse map for $T$ by $V.$

Extending $T$ and $V$ to  (multiplicative) unitaries on $\HtH$  shows that $V$ is continuous with respect to the Hilbert space norm topology on $\HtH.$
Consider a sequence $(x_i)$ of algebraic elements of $A\tensor A$ that converges in the (minimal) C*-tensor product  norm on $A\tensor A$ to some element of $A\mintensor A.$ Thus,~$Vx_i$ converges in this topology to some element $L$ of $\HtH.$  We wish to show that this element $L$ is actually in  $A\tensor\H.$
Thus~we must compute the Hilbert space coefficients of $L$ with respect to the second factor of this tensor product---in other words, apply to $L$ the slice maps $\Id\tensor f(g_j^* \cdot),$ where the $g_j$ are the basis of $\H$ that we constructed earlier.

We first show that the $Vx_i$ are algebraic elements, in $A_0 \tensor A_0.$ When applied to an algebraic element of $A_0\tensor A_0,$ the map $T$ takes $a\tensor b$ to $\coproduct{}(a) (b\tensor 1).$ By  \cite{VanDaele1994} (Proposition 2.1), applied in $A^{op},$ 
 it~follows that this map has the algebraic inverse $V:=  \sigma\circ(r\tensor\Id)(\Id\tensor\antipode\tensor\Id)(\Id\tensor\coproduct{}),$
where $r$ is the bi-linear map that takes a simple tensor $a\tensor b$ to the product $ba,$ and $\sigma$ is the tensor flip. Since the $x_i$ are algebraic elements of $A\tensor A$, it follows that the sequence $(Vx_i)$ consists of algebraic elements of $A\tensor A.$

 We are now in a position to compute the required coefficients. We recall that the $g_j$ are algebraic elements of $A.$ Applying the slice map $\Id\tensor f(g_j^* \cdot)$ to $V,$ we have
\begin{equation*} \left(\left(\Id\tensor f(g_j^* \,\cdot)\right)V\right)= (f(g_j^* r(\cdot)\tensor\Id)(\Id\tensor\antipode\tensor\Id)(\Id\tensor\coproduct{}),\label{eqn:fV}
\end{equation*}
where the linear functional $f(g_j^* r(\cdot))\circ(\Id\tensor\antipode)$ can also be written, using properties of the pairing $\beta,$~as
\begin{equation}G_{j}\colon b\tensor c \mapsto \beta(b\tensor c,(\Id\tensor\antipodeh)\circ\sigma\circ\coproduct{}(\F(g_j^*))\,),\label{eqn:pairing}
\end{equation}
 where $\F(g_j^*)$ is the Fourier transform of the element $g_j^*,$ and $\sigma$ is the tensor product flip.
Equation~\eqref{eqn:pairing}  provides, for each fixed basis element $g_j\in \H,$ a bounded linear functional, $G_{j},$ on the whole of the C*-algebraic  tensor product $A\tensor A.$ The boundedness of $G_{j}$ and the fact that the sequence $x_i$ was chosen to converge in the C*-norm implies that
\begin{equation} (\Id\tensor f(g_j^* \cdot))Vx_i  = (G_{j}(\cdot)\tensor\Id)(\Id\tensor\coproduct{})x_i ,\label{eqn:gV}
\end{equation}
{converges, with respect to $i,$ in the C*-norm on $A$.
 Thus, the limit is an element of the C*-algebra $A$. This shows that the Hilbert space coefficients of $L$ are in $A,$ and therefore  that $L$ is in $A\tensor\H.$ 
This means that $V$ maps  $A\mintensor A$ to $A\tensor\H.$ Since $V$ is inverse to $T,$ and $T$ is a 2-isometry with the property $T(ax)=\coproduct{A} (a)T(x),$ it follows that $V$ is a 2-isometry with the property $V(\coproduct{A}(a)x)=aV(x).$}
An~alternative argument to show that $L$ is in $A\tensor\H$, using metrics rather than coefficients, is, following~\cite{halmos} (pp.~182--183): since we have assumed that $A$ is separable, the weak topology of $\H$ is metrizable on bounded subsets. The sequence $Vx_i$ can be supposed to be in the unit ball of the Hilbert space $\HtH,$ and  then Equation \eqref{eqn:gV} implies that $Vx_i$ converges to $L$ in $A\tensor\H$ with the  C*-norm on the first factor and the metric $\sum \frac{2^{-j}}{||G_j||}|<g_j,\cdot>|$ on the second factor. 
\end{proof}
  An operator $V$ having the property $V(\coproduct{}(a)x)=aV(x)$ will be referred to as having the \textit{twisted-linear} property. The twisted-linear property implies ordinary linearity over $\C,$ because the~coproduct $\coproduct{}$ is a unital homomorphism and maps $\lambda 1_A $ to $\lambda\coproduct{}(1_{A})$ for all complex scalars $\lambda.$ The~2-isometry $V$ is, from the C*-algebraic point of view, a (restriction of an) unitary, but when using the twisted-linear property, it seems more natural to denote the inverse by $V\inv$ instead of $V^*.$

The usual GNS construction provides an inner product module.  The completion of $A\tensor \H$ in the norm provided by the inner product is (unitarily equivalent to) the standard Hilbert module, denoted $\HA.$ {Some authors, such as Lance, denote the standard Hilbert module by $H\tensor A.$}  Thus, the above operator $V$ maps $A\tensor A$ into $\HA.$ The ``compact'' operators $\compact(\HA)$ are isomorphic to $A\tensor\compact$ and the adjointable bounded operators $\adjointable(\HA)$ are isomorphic to the multiplier algebra $\Mult(A\tensor\compact).$ See \cite{Kasparov.modules,lance}. 

We now take the first step towards defining a product on the Cuntz semigroup:
\begin{Definition} If $M_1$ and $M_2$ are (right) Hilbert sub-modules of $A$, we denote by $M_1\boxproduct M_2$ the (right) Hilbert sub-module of $\H_A$ obtained by closing
$V(M_1\tensor M_2)$ in the Hilbert module norm on $\H_A.$
\label{def:boxproduct}\end{Definition}

The main issue is to show that this product operation is well-behaved at the level of Cuntz equivalence classes of Hilbert modules. It will follow that the Cuntz semigroup of $A$ is a (semi)ring.

The 2-isometry property can be strengthened in the case of $A$-module maps:
\begin{Proposition} Suppose $A$ is a compact-type C*-algebraic quantum group with faithful Haar state $f$ and that $E_1$ and $E_2$ are inner product modules over $A.$ Suppose that $V\colon E_1\arrow E_2$ is an $A$-module map that is  2-isometric with respect to $f.$ The map $V$ is then isometric with respect to the norm on the inner product module.\label{prop:2isometries.want.to.be.isomorphisms}
\end{Proposition}
\begin{proof}
We claim that if $y$ is a positive element of $A,$ then the C*-norm of $y$ is given by the formula:
$\|y\|_{A}=\sup_{b\not=0} \frac{|f(b^*yb)|}{f(b^*b)},$ where the supremum is over the nonzero elements $b$ in $A,$ and $f$ is the given faithful state. From the $A$-linearity of $V$ and the hypotheses, we will then have that
$ f\left(b^*\ip{x}{x}_{E_1}b\right)= f\left( \ip{b x}{bx}_{E_1}\right) =f\left(\ip{Vb x}{Vbx}_{E_2}\right) = f\left(b^*\ip{Vx}{Vx}_{E_2}b\right).$
Dividing the left side and the right side of the above by $f(b^*b),$ and then applying the claim, we have that $\|\ip{x}{x}_{E_1}\|_{A}=\|\ip{Vx}{Vx}_{E_2}\|_{A},$ from which it follows that
$\|x\|_{E_1}^2=\|Vx\|_{E_2}^2.$ In other words, $V$ is  isometric with respect to the usual norms on $E_1$ and $E_2,$ as was to be shown. This argument uses the completeness of the algebra, $A$, but does not require the module, $E_i,$ to be complete. Thus, the conclusion holds in inner product modules as well as in Hilbert modules.

The claim remains to be established. 
 Given a positive element $y\in A$, we have
$\|y\|_{A} =\sup\frac{|h(y)|}{\|h\|}$ where the supremum over nonzero $\sigma$-weakly continuous positive linear functionals $h$ and $\|h\|$ denotes the usual operator norm of a linear functional. By Sakai's quadratic Radon--Nicod\'ym theorem \cite{Sakai1965}---see also~\cite{niestegge} (Corollary 1.6 and Theorem 2.6)---it follows that linear functionals of the form $h=f(b^*\cdot b)$ with $b$ coming from the algebra are dense within the  linear functionals in our class. (This is well-known, see, e.g., van Daele's theory of multiplier Hopf algebras \cite{VanDaele1994}.)
We thus have that
$\|y\|_{A}=\sup_{b\not=0} \frac{|f(b^*yb)|}{\|f(b^*\cdot b)\|},$ where the norm in the denominator is the usual operator norm. The operator norm of a positive linear functional is equal to the value of the functional at $1,$ so that $\|f(b^*\cdot b)\|=f(b^*b).$
\end{proof}

Following  \cite{MF} (Definition 3.2), two Hilbert $A$-modules are said to be \textit{isomorphic as Hilbert C*-modules} if and only if there exists a linear bijective mapping $F$ of one onto the other such that $F(ax)=aF(x)$ and $\ip{x}{y}_{E_1}=\ip{Fx}{Fy}_{E_2}$.
This form of isomorphism is well-known to be equivalent to several other slightly different forms of isomorphism, and thus we can push the result of the last proposition further. The next Theorem is a straightforward corollary of the last proposition, and can be seen as a kind of Mazur--Ulam theorem \cite{MazurUlam,MazurUlam2} for Hilbert modules. It gives conditions under which a mapping that is an isometry in a certain weak sense is in fact a unitary from one Hilbert module to another: unlike the classical Mazur--Ulam theorem, we definitely need to assume $A$-linearity.
 
\begin{Theorem}Let $F\colon E_1 \arrow E_2$ be a surjective $A$-linear map between Hilbert $A$-modules, where $A$ is a C*-algebraic quantum group  with faithful state, $f.$ Then, if the map has any one of the following four properties, it has all of them: \begin{enumerate}
\item $F$ is a unitary element of the bounded adjointable operators $\mathcal{L}(E_1,E_2)$,
\item $F$ is a Hilbert module isomorphism,
\item $F$ is an isometry with respect to the norms on $E_1$ and $E_2,$ and
\item $F$ has the property $ f\left(\ip{x}{y}_{E_1}\right)= f\left(\ip{Fx}{Fy}_{E_2}\right),$ where $f$ is the given faithful state.
\end{enumerate}\label{th:equivalences.of.maps}
\end{Theorem}
\begin{proof} It is clear that the first property implies all the other properties. That the second property implies the first means showing that a mapping providing an isomorphism of Hilbert  $A$-modules is automatically adjointable. This is shown in  \cite{lance1994} (Theorem 1). See also the more recent reference~\cite{lance} (Theorem 3.5.)  For the proof that the third property implies the second, we use a result of Blecher's~\cite{blecher} (Theorem 3.2), re-interpreted somewhat \cite{MF} (Proposition 3.3), which shows that a surjective  $A$-linear Banach module map $F$ of Hilbert $A$-modules that has the properties $F(ax)=aF(x)$ and $\|x\|_{E_1}=\|F(x)\|_{E_2}$ is automatically an isomorphism of Hilbert modules.  That the third property is implied by the fourth is Proposition \ref{prop:2isometries.want.to.be.isomorphisms}.
\end{proof}

The following Proposition collects some known results, valid in the case of stable rank 1:
\begin{Proposition} Let $A$ be a separable C*-algebra of stable rank 1. Then,
\begin{enumerate} 
\item The order relation in the Cuntz semigroup is an inclusion of Hilbert modules, and
\item The equivalence relation in the Cuntz semigroup is  Hilbert module isomorphism.
\item Every element of $\Cu(A)$ is the supremum of a sequence of elements from $W(A),$
\item Every increasing sequence of elements of $\Cu(A)$ has a supremum in $\Cu(A),$ and
\item At the level of ordered monoids, $\Cu(A)$ is the completion of $W(A).$
\end{enumerate}\label{prop.completions}
\end{Proposition}
\begin{proof} Theorem 5.1 and Lemma 2.10 from  \cite{ABP} show that in the stable rank 1 case,  $\Cu(A)$ is the completion of $W(A).$ Definition 3.1. iii of  \cite{ABP} then gives that every element of $\Cu(A)$ is the supremum of a sequence of elements from $W(A).$ That every increasing sequence of elements of $\Cu(A)$ has a~supremum is Theorem 1 in  \cite{CEI}, which incidently does not assume stable rank 1. Theorem  3 in  \cite{CEI} gives the order relation and equivalence relation in the stable rank 1 case.
\end{proof}

\begin{Theorem} In a compact-type C*-algebraic quantum group $A$ that is separable, nuclear, and of stable rank 1, if $M_1,$ $M_2$ and $M_3$ are Hilbert sub-modules of $A$, and $M_1$ is equivalent in the Cuntz semigroup to $M_2,$ then $M_1\boxproduct M_3$ is equivalent in the Cuntz semigroup to  $M_2\boxproduct M_3.$
\label{th:invariant}
\end{Theorem}
\begin{proof}  The equivalence relation on the Cuntz semigroup can be taken to be, in the special case of stable rank 1, simply isomorphism classes of Hilbert  modules, and the order relation to be inclusion of Hilbert modules. We consider $W:=V(M_1\tensor M_3)$ which is an inner product sub-module of $A\tensor\H$, and  $W':=V(M_2\tensor M_3),$ an inner product sub-module of $A\tensor\H.$ We have, from the hypothesis, that
$M_1\tensor M_3$ is isomorphic as a Hilbert $(A\tensor A)$-module to $M_2\tensor M_3.$ Denoting this isomorphism $F,$ we~compose  with the twisted-linear maps $V\inv$ and $V$ of Lemma \ref{lem:special.module},  obtaining $V\compose F\compose V\inv\colon W\arrow W'.$

Summarizing these considerations in a diagram, we have:
\begin{equation} \begin{diagram}[small]
M_1\tensor M_2  & \lTo^{F} & M_1\tensor M_3 \\
 \dTo^{V} &        &\dTo_{V} \\
{V(M_1\tensor M_2 )} & \lDotsto_{G} & {V(M_1\tensor M_3 )} \\
\end{diagram}\label{diag:inclusions.in.low.rank.case}\end{equation}
where the solid arrows are all  2-isometric maps. Thus, the map $G$ defined by the dotted arrow above is a 2-isometric map. It is an $A$-module map by the twisted-linear property of $V.$  By Proposition~\ref{prop:2isometries.want.to.be.isomorphisms}, this~implies that the map $G$ is isometric with respect to the norm on the inner product modules, which is the norm coming from the enveloping Hilbert module. If we take the closure of  ${V(M_1\tensor M_3 )}$ and  ${V(M_1\tensor M_3 )}$ with respect to the Hilbert module norm on  $\HA,$ the fact that $G$ is an~isometry in this norm allows us to extend $G$ by continuity to an isometry of the closures. We~now have an isometry $G\colon\overline{V(M_1\tensor M_2 )} \arrow  \overline{V(M_1\tensor M_3 )}.$ Theorem \ref{th:equivalences.of.maps} then gives a Hilbert module isomorphism of these Hilbert $A$-modules.
\end{proof}

\begin{Corollary} The operation $\boxproduct$ gives an associative product on the Cuntz semigroup of a separable compact C*-algebraic quantum group with faithful Haar state and stable rank 1. \label{cor:extend.product}\end{Corollary}
\begin{proof}Theorem \ref{th:invariant} shows that the operation $\boxproduct$ is well-defined at the level of the Cuntz semigroup, at least for Hilbert C*-modules that are sub-modules of $A.$
In the proof of Theorem \ref{th:invariant}, we replace $A$ by a matrix algebra over $A,$ and $V$ by $1\tensor V,$ obtaining a product on $W(A).$  (We consider  $1\tensor 1\tensor V\colon M_n \tensor M_n \tensor A \tensor A \rightarrow M_{n^2} (A) \tensor H$, and then flip the two middle factors of  the first tensor product, obtaining a map from $M_n (A) \tensor M_n(A)$ to $M_{n^2} (A) \tensor H.$)
 We note that the product operation respects inclusion, in each variable separately. Extension of the product to $\Cu(A)$ thus follows by passage to order-completions, in each variable separately.
In more detail, suppose that $x$ is an element of $W(A)$ and $y$ is an element of $\Cu(A).$ Proposition \ref{prop.completions}.iii gives an increasing sequence $y_n$ of elements of $W(A)$ that has supremum $y.$ Then, $x\boxproduct y_n$ is an increasing sequence of elements of $Cu(A),$ and the supremum of this sequence, which exists by  Proposition \ref{prop.completions}.iv,  defines $x\boxproduct y.$ At the level of Hilbert modules, we~view $x\boxproduct y_n$ as a submodule of $\HA$, and then take the closure of the union of the submodules $ x\boxproduct y_n$ with respect to the Hilbert module norm.
Then, we perform a similar process with respect to the other~variable.

 Associativity, with respect to multiplication, follows  from co-associativity of the co-product.  It is sufficient to consider the case of $W(A).$  If we describe the $A$-modules $M_1\boxproduct (M_2 \boxproduct M_3)$ and $(M_1\boxproduct M_2) \boxproduct M_3$ purely algebraically, in both cases we obtain the same set of algebraic generators but with scalar actions given by $(\Id\tensor\coproduct{})\compose\coproduct{}(a)$  and $(\coproduct{}\tensor\Id)\compose\coproduct{}(a)$ respectively. However, then the co-associativity of the coproduct gives that these are the same module. We thus obtain an $A$-module map from $M_1\boxproduct (M_2 \boxproduct M_3)$ onto $(M_1\boxproduct M_2) \boxproduct M_3,$ and since $V$ is a 2-isometric map, the map obtained is a composition of  2-isometric maps; hence it is itself 2-isometric, and thus a Hilbert module isomorphism (Theorem \ref{th:equivalences.of.maps}.)
\end{proof}

\begin{Remark} Elements of $W(A)$ can always be written, in the operator picture, as diagonal matrices with coefficients in $A.$ The properties of the tensor product mapping $M_n \tensor M_n\rightarrow  M_{n^2}$ used in the above proof give in the case of diagonal matrices:
$\left(\begin{matrix}a&\\&b\end{matrix}\right)\boxproduct\left(\begin{matrix}c&\\&d\end{matrix}\right) = \left(\begin{smallmatrix}a\boxproduct c\\&a\boxproduct d\\&&b\boxproduct c\\&&&b\boxproduct d\end{smallmatrix}\right). $ In an additive notation:  $\left(a+b\right)\boxproduct\left(c+d\right) = a\boxproduct c+a\boxproduct d+b\boxproduct c+b\boxproduct d. $\label{rem:distr}
 \end{Remark}
Thus, we have a distributive property for the product $\boxproduct$ with respect to the additive structure of $W(A),$ and since we extended to $\Cu(A)$ by taking completions at the monoid level, the product on  $\Cu(A)$ inherits this property. We note that the product map with respect to a fixed element $M$ gives a map from $\Cu(A)$ to $\Cu(A),$ defined by $x\mapsto M\boxproduct x,$  that is well-behaved with respect to the semigroup structure. Moreover, the map given by taking a product with a fixed element respects the property of being a sub Hilbert module: in other words, if $x$ is contained in $y$ then $M\boxproduct x$ is contained in $M\boxproduct y.$ 
 When applying the map to an increasing sequence of Hilbert modules, we obtain an increasing sequence of Hilbert modules, which however are not larger than the standard module, $\H_A $.  In the proof of Corollary \ref{cor:extend.product}, see the end of the first paragraph of the proof, it is shown that, at least in the stable rank one case, the product map with respect to a fixed element will preserve the suprema of increasing sequences. Clearly, the product map will preserve the zero of the Cuntz semigroup. Thus, we see that the product map with respect to a fixed element gives a well-behaved mapping of Cuntz semigroups, namely a morphism of Cuntz semigroups in the terminology of \cite{CEI}.  (The general definition of a~morphism does involve compact containment, however, in the stable rank one case, it is not necessary to consider compact containment of Hilbert modules.)

We note that we can obtain a left and a right product structure, one associated with the left Haar state and the other associated with the right Haar state. It is not clear if there exist further variations on the product than this. On the other hand, the fact that we can sometimes, as will be shown, use the product structure to obtain isomorphism results for C*-algebraic quantum groups suggests that the product has a certain canonical status. 

\begin{Remark}  In the finite-dimensional case, the Cuntz semigroup simplifies considerably, and our product reduces to the simpler product considered in \cite{kuce.classify.Hopf.by.Ktheory0}. This product is defined by
$\coproduct{*} (M_1 \otimes M_2) $
where $M_1$ and $M_2$ are algebraic modules over $A$, and $\coproduct{*}$ is the algebraic restriction of rings operation
 induced by the coproduct homomorphism.
This provides an interesting example of the product we~have just defined.\label{rem:multiplication.table} In the finite-dimensional case, the class of the support projection of the co-unit character will act as a unit for the product we~have introduced. In~general, no such support projection need exist. We~will eventually  see, from  Proposition \ref{prop:box.and.convolve} and the properties of convolution, that if $\Id$ is~the multiplicative identity of $A,$ then
$[\Id]\boxproduct [a] $ is equivalent to $\Id$ for all positive~$a$~in~$A.$ \end{Remark}
We have, in this section, added analytic structure to the algebraic restriction of rings operation from Remark \ref{rem:multiplication.table}. Another way to implement a restriction of rings operation at the level of Hilbert modules is via Kasparov's inner tensor product \cite{Kasparov.modules}. This tensor product does not behave well functorially, though in some cases  Hilbert module connections (introduced by Connes and Skandalis~\cite{CS}) can be used instead. The connection technique does not seem applicable in our case.

\section{Higher Stable Rank}\label{Section3}

In this  section, we generalize Theorem \ref{th:invariant} to the case of higher stable rank. As already pointed out, it is desirable for applications to reformulate our results in the operator picture of the Cuntz semigroup, and we do that in this section as well. 
The basic step in passing from the Hilbert module picture to the operator picture of the Cuntz semigroups is that we can  work in terms of generators of Hilbert modules rather than the modules themselves. If we restrict attention to Hilbert modules over $\sigma$-unital C*-algebras, given a Hilbert module, we may in general choose a strictly positive element of the ``compact'' operators upon it as a generator. This fact is often summarized by saying that Hilbert modules over separable or $\sigma$-unital C*-algebras can be taken to be singly generated; see \cite{CEI} (p. 39). To recover the Hilbert module from this strictly positive element, $\ell$, we show that
the element belongs to the compact operators on the Hilbert module $\H_A,$ and then consider $\overline{\ell H_A}.$ There apparently remains a slight gap between the Hilbert module picture of the Cuntz semigroup and the operator picture, in~which the elements of the Cuntz semigroup $\Cu(A)$ lie in $A\tensor\compact.$ The issue is  that in one case, we have ``compact'' operators on $\H_A$, and in the other case, we have elements of $A\tensor\compact.$ However, Kasparov has shown \cite{kasparov1981} (Section 1, Paragraph 14) that these two objects are in fact isomorphic. In the case of sub-Hilbert modules of $A,$ the generator obtained is of course a positive element of the $C^*$-algebra $A,$ and we recover the module by considering the one-sided closed right ideal $\overline{aA}$ generated by the given algebra element, following the conventions in \cite{CEI}. Thus, we have the operator picture of the Cuntz semigroup:  for more information see \cite{CEI} (Appendix 6).

 We say that $a$ is Cuntz subequivalent to $b,$ denoted $a \leqCu b$, if there is a sequence $x_n$ such that $x_n^* b x_n$ goes to $a$ in norm. Thus, for example, $e^*xe\leqCu x.$ We~say that $a$ and $b$ are Cuntz equivalent if we~have both $a \leqCu b$ and $a \geqCu b.$

We recall an equivalence relation on positive elements of a C*-algebra, $A,$ that was considered by Blackadar  \cite{blackadar88}, and is denoted here by $a\blackadar b.$ This equivalence relation is generated by the following two equivalence relations:
 
\begin{enumerate} 
\item positive elements $a$ and $b$ are equivalent in a C*-algebra $A$ if they generate the same hereditary subalgebra of $A,$ and
\item
positive elements $a$ and $b$ are equivalent in a C*-algebra $A$ if there is an element $x\in A$ such that $a=x^*x$ and $b=xx^*.$ \end{enumerate}

The following known Lemma relates the equivalence relation $\blackadar$ on positive elements to properties of the Hilbert submodules of $A$ that are generated by the given positive elements:
\begin{Lemma}[\protect{\cite{ORT} ([Proposition \!4.3])}] Let $A$ be a C*-algebra, and let $a$ and $b$ be positive elements of $A.$ The~following are equivalent:
\begin{enumerate}
 
\item $a\blackadar b,$ and
\item $\overline{aA}$ and $\overline{bA}$ are isomorphic as Hilbert $A$-modules.
\end{enumerate}\label{lem:isos}
\end{Lemma}
A relationship between the equivalence relation $\blackadar$ and the Cuntz equivalence relation is given by  the following known Lemma. We recall that the function  $f(\lambda):=\left(\lambda-\tfrac{1}{n}\right)_{+}$ is the function that is zero on $(-\infty,\tfrac{1}{n})$ and is equal to $\lambda-\tfrac{1}{n}$ elsewhere. The functional calculus for C*-algebras lets us apply this function to a positive or self-adjoint element of a C*-algebra.
\begin{Lemma} Let $A$ be a C*-algebra. The following are equivalent:
\begin{enumerate}
\item  $a \geqCu b,$ and
\item for all $n\in\N,$ we have  $\left(b-\tfrac{1}{n}\right)_{+}\blackadar d_n \in \overline{aAa}.$\item for all $n\in\N,$ we have  $\left(b-\tfrac{1}{n}\right)_{+}\blackadar d_n,$ where $d_n$ has the factorization $f_n(a)a_n f_n(a),$ where $a_n$ is an~element of  $A$ and $f_n$ is some function in $C_0 (0,1],$ bounded above by $1.$
\end{enumerate}\label{lem:from.blackadar.to.cuntz}
  \end{Lemma}
\begin{proof} The equivalence of \textit{1} and \textit{2} follows from Lemma 2.4. \textit{iv} in  \cite{kirchberg.rordam2000}. The equivalence of \textit{2} and \textit{3} follows from Cohen's factorization theorem  \cite{cohen}.\end{proof}

We also recall a known continuity property possessed by the Cuntz subequvalence relation:
\begin{Lemma}If $x_n$ converges to $x$ in norm, and $x_n\leqCu y$ then $x\leqCu y.$
\label{lem:cont}\end{Lemma}

\begin{Lemma} Given a submodule $M$ of the standard Hilbert module, $\HA,$ not necessarily closed, there exists a~positive element $e$ of the ``compact'' operators on $\HA,$ such that $\overline{eM}=\overline{M},$ where the closure is taken in the topology coming from the Hilbert module norm. \label{lem:weakCC}\end{Lemma}
\begin{proof} Since $\HA$ is countably generated, so is $\overline{M}.$ Therefore, the ``compact'' operators $\mathcal K_{\overline{M}}$ are $\sigma$-unital, and a strictly positive element of $\mathcal K_{\overline{M}}$ will have the desired property, except that it is apparently not an element of  $\mathcal K_{\HA}.$ However, Kasparov's stabilization theorem \cite{Kasparov.modules} implies that $\overline{M}$ is, after unitary equivalence, a direct summand in  ${\HA}$ and thus that $\mathcal K_{\overline{M}}\subseteq \mathcal K_{\HA}.$\end{proof}
Observe that in the proof of Theorem \ref{th:invariant}, the hypothesis of stable rank 1 is used only to simplify the equivalence relation in the Cuntz semigroup to  Hilbert module isomorphism. If we instead state the Theorem in terms of isomorphism of Hilbert modules, we can drop the hypothesis of stable rank 1.
Thus, Lemma  \ref{lem:isos} and Theorem  \ref{th:invariant} imply the Corollary:

\begin{Corollary} In a compact-type C*-algebraic quantum group $A$ that is separable and  nuclear, with a faithful Haar state, if $a,$ $a'$ and $b$ are positive elements of $A$, and $a\blackadar a',$ then a positive operator generating the Hilbert $A$-module $[a]\boxproduct [b]$ is equivalent under $\blackadar$ to any positive operator generating the Hilbert $A$-module  $[a']\boxproduct[b].$
\label{cor:invariant.blackadar} \end{Corollary}
The above Corollary says that the product operation is compatible with the equivalence relation~$\blackadar$. We now extend the Corollary to the Cuntz equivalence relation, which is in general a coarser equivalence relation than $\blackadar.$

\begin{Corollary} In a compact-type C*-algebraic quantum group $A$ that is separable and  nuclear, with a faithful Haar state, if $a,$ $a'$ and $b$ are  positive elements of $A$,  and $a$ is Cuntz equivalent to $a',$ then $[a]\boxproduct [b]$ is Cuntz equivalent to $[a']\boxproduct[b].$
\label{cor:invariant.higher.Cuntz.equiv} \end{Corollary}
\begin{proof}
Suppose that $a'\geqCu a.$  We will show that $[a']\boxproduct [b] \geqCu [a]\boxproduct [b].$
By Lemma \ref{lem:from.blackadar.to.cuntz}
we have, for $\varepsilon>0,$ that $\left(a-\varepsilon\right)_{+}\blackadar d_n,$ where $d_n$ is in the hereditary subalgebra within $A$ generated by $a'.$ At the level of Hilbert modules, $[d_n]$ is thus a submodule of the Hilbert module $[a']$ generated by $a'.$
Corollary \ref{cor:invariant.blackadar}, together with the already established fact that the product $\boxproduct$ respects (in each variable separately) the ordering of Hilbert modules given by inclusion, now implies that  $[\left(a-\varepsilon\right)_{+}]\boxproduct [b] $ is isomorphic to a~submodule of $[a']\boxproduct [b].$ However, then
\begin{equation}[\left(a-\varepsilon\right)_{+}]\boxproduct [b]\leqCu [a']\boxproduct [b].
\label{inequality1}
\end{equation}

We now aim to use Lemma \ref{lem:cont} to deduce from the above that $[a]\boxproduct [b]\leqCu [a']\boxproduct [b].$

According to Definition \ref{def:boxproduct}, the product $[a] \boxproduct [b]$ is the closure of $V(a\tensor b)(A\tensor A)$ in $\mathcal H _A .$ Lemma \ref{lem:weakCC} lets us insert a positive element $e\in \mathcal K (\mathcal H _A ),$ obtaining   $\overline{eV(a\tensor b)(A\tensor A)}.$
It follows that the Hilbert module $[a]\boxproduct [b]$ is generated by
the positive operator
$eV(a \tensor b)V\inv e \in \mathcal K ( \HA ).$ (The positivity comes from the fact that $V$ extends to a unitary.) Replacing $a$ by $(a-\varepsilon)_+$ we consider $$x_\varepsilon = eV((a-\varepsilon)_+ \tensor b)V\inv e.$$

For each $\varepsilon>0$, this is a positive operator in $\mathcal K (\HA )$
that   is Cuntz subequivalent 
to $[\left(a-\varepsilon\right)_{+}]\boxproduct [b].$ 
By Equation \eqref{inequality1}, this is in turn Cuntz subequivalent to  $[a']\boxproduct [b].$

The operator $x_\varepsilon$ depends norm-continuously on $\varepsilon.$ Applying Lemma \ref{lem:cont} then shows that the limit of the $x_\varepsilon$ is Cuntz subequivalent to $[a']\boxproduct [b].$ Thus, $[a]\boxproduct [b] \leqCu [a']\boxproduct [b].$ This shows that the product respects Cuntz subequivalence. The other subequivalence follows similarly, and thus the product respects Cuntz equivalence.
\end{proof}
The above Corollary shows that we have a well-defined product operation on elements of the Cuntz semigroup that come from elements of $A$.  We  extend to the stabilization, as expected:

\begin{Corollary} The operation $\boxproduct$ gives a product on the Cuntz semigroup $\Cu(A)$ of a separable compact-type C*-algebraic quantum group $A$ with faithful Haar state. The product distributes over finite sums of elements of~$W(A).$ \label{cor:existsProduct}
\end{Corollary}
\begin{proof} Corollary \ref{cor:invariant.higher.Cuntz.equiv} shows that the operation $\boxproduct$ respects the Cuntz equivalence relation and subequivalence relation, at least when restricted to submodules of $A.$
Replacing $A$ by the stabilization of $A$ lets us obtain unrestricted elements of the Cuntz semigroup. In the proof of Corollary \ref{cor:extend.product}, we~replace $A$ by $\compact\tensor A$ and $V$ by $1\tensor V,$ obtaining a product on $\Cu(A).$ Associativity follows as in that proof, and distribution over finite sums follows as in Remark \ref{rem:distr} \end{proof}

\section{Hopf Algebra Maps from Cuntz Semigroup Maps: The Operator Picture and the Open Projection Picture}\label{Section4}

In this section, we show that algebra isomorphisms that induce maps on the Cuntz semigroup which respect the product, are in fact Hopf algebra isomorphisms or co-anti-isomorphisms.
Algebra maps that induce maps on the Cuntz semigroup that respect the product $\boxproduct$ will be called \textit{\/$\fancyK$-co-multiplicative maps.}
There is a Fourier transform on $A$ --- see  \cite{kahng} (Definition 3.4) --- that can be defined as $\F(b):=[\haar(\cdot b)\tensor\Id] V ,$ where $V$ is a multiplicative unitary coming from the left regular representation, and $\haar$ is the (extension of the) left Haar state.  In terms of  linear functionals, this  Fourier transform is given by, following  Van Daele \cite{VanDaele1994}, see also  \cite{kahng}, by
$\beta(a,\F(b))=\haar(ab),$ where  the elements $a$ and $b$ belong to a C*-algebraic quantum group $A,$ $\haar$ is the left Haar weight,  and $\beta(\cdot\,,\cdot)$ is the pairing with the dual algebra.
Following \cite[]{kahng} (Definition 3.10) and the discussion there, a well-behaved  operator-valued convolution product, $\convolve,$ can be defined by the  property $\F(a\convolve b)=\F(a)\F(b),$ where $a$ and $b$ are elements of a C*-algebraic quantum group $A,$ and $\F$ is the (unbounded) Fourier transform defined previously.
The convolution operation takes a pair of positive operators to a positive operator.

 We now assume tracial Haar state in order to establish a relationship between the operations $\boxproduct$ and~$\convolve.$ 

At the start of Section \ref{Section3}, there is a discussion showing  that elements of the Cuntz semigroup $\Cu(A)$ can be (non-uniquely) given as ``compact'' operators (in $A\tensor\compact$) which map $A\tensor\H$ to $A\otimes \H.$ The next Proposition gives a  choice of such an operator for the product $[\ell_1]\boxproduct[\ell_2]$.

\begin{Proposition} Let $A$ be a compact-type C*-algebraic quantum group with faithful tracial Haar state.
Let~$\ell_1$ and $\ell_2$ be positive elements of $A.$
We may represent the product $[\ell_1]\boxproduct[\ell_2]$ of Theorem \ref{th:invariant}, in the operator picture of the Cuntz semigroup, by the positive operator $V(\ell_1\tensor \ell_2)V\inv\colon A\tensor\H\arrow A\tensor\H.$ (This operator is ``compact'' in the usual Hilbert module sense.)
 We have $(\Id\tensor t) V(\ell_1\tensor\ell_2)V\inv=\ell_1 \convolve \ell_2$ for all positive $\ell_i\in A,$ where  $t$ is the standard trace on $\compact.$\label{prop:box.and.convolve}
\end{Proposition}
\begin{proof} Since ${\ell_1 A} \boxproduct {\ell_2 A}$ is defined to be the closure of $V(\ell_1\tensor\ell_2)(A\tensor A),$ and since $V\inv (A\tensor\H)=A\tensor A$, we have that  ${\ell_1} \boxproduct {\ell_2}$ is the closure of the (pre) Hilbert module $V(\ell_1\tensor\ell_2)V\inv (A\tensor\H).$

Now suppose that the Haar state is tracial.
  Regarding $\ell_i$ as being represented on the GNS Hilbert space coming from the Haar state on $A,$ we note that the Haar state extends to the canonical (unbounded) trace on $\BH,$ and that $\ell_i$ is then of trace class. We will denote both the Haar state and its extension to $\BH$ by $\haar.$ Since $V\colon A\tensor A \arrow A\tensor\H$ extends to a unitary on $\H\tensor\H,$ it follows that
$V(\ell_1\tensor \ell_2)V\inv$ is trace class, and positive, on the tensor product $\H\tensor\H.$ Being trace class under this representation  implies that $V(\ell_1\tensor \ell_2)V\inv,$ as a Hilbert module operator, is not just in the bounded adjointable operators, $\adjointable( \HA ),$ but is actually in the ideal of ``compact'' operators $\mathcal K(\HA)\isom A\tensor\compact.$
Applying the slice map  $\Id\tensor\haar$ to $V(\ell_1\tensor \ell_2)V\inv,$ we have an element of $A.$ Let $L:=(\Id\tensor\haar)V(\ell_1\tensor \ell_2)V\inv.$

We next show $\F(L)=\F(\ell_1)\F(\ell_2).$ If we apply $\F(a)$ to $L$ and use the property that $V$ maps the action of $a$ on $A\tensor \H$ to the action of $\coproduct{}(a)$ on $A\tensor A,$ we have
$ \haar(a\cdot)\compose  (\Id\tensor\haar)V(\ell_1\tensor\ell_2)V\inv= (\haar\tensor\haar)(V\coproduct{}(a)(\ell_1\tensor\ell_2)V\inv).$
Since $\haar\tensor\haar$ is a tensor product of traces, it is a trace, and the right hand side of the above is equal to $(\haar\tensor\haar)(\coproduct{}(a)(\ell_1\tensor\ell_2)).$ If we use the definition of the Fourier transform, $\beta(a,\F(b))=\haar(ab),$ to replace tracial Haar states by pairings with the Fourier transform, we can simplify this to $\beta(\coproduct{}(a),\F(\ell_1)\tensor\F(\ell_2)).$ However, this last expression is equal to
$\beta(a,\F(\ell_1)\F(\ell_2)).$ Thus, we have shown that $\haar(aL)=\beta(a,\F(\ell_1)\F(\ell_2)),$ and since $\haar(aL)$ is equal to $\beta(a,\F(L)),$ the nondegeneracy of the pairing implies that $\F(L)=\F(\ell_1)\F(\ell_2).$ Thus, $L=\ell_1 \convolve \ell_2,$ where $L=(\Id\tensor\haar) (\ell_1\boxproduct \ell_2).$ This~proves the result, with $\haar$ instead of $t,$ but as already mentioned, in this representation
 the extended Haar state $\haar$ coincides with the usual unbounded trace $t.$
\end{proof}
The above Proposition shows that the Hilbert module $[a]\boxproduct[b]$ is generated by an operator having a certain form, and relates this operator to convolution, at least in the case that the Haar state is tracial. At present, the second statement, $(\Id\tensor t) V(\ell_1\tensor\ell_2)V\inv=\ell_1 \convolve \ell_2,$ does not pass to Cuntz equivalence classes, because applying the slice map does not respect Cuntz equivalence. So at the moment this second statement is simply an identity at the level of operators.

\subsection*{The Open Projection Picture}
The open projection picture presents the Cuntz semigroup in the form of projections together with an equivalence relation and states coming from traces, rather like  
 in the case of a K-theory group. Although the equivalence relation is not the K-theory equivalence relation---and the ambient algebra is the double dual---there is nevertheless a parallel between Cuntz semigroups and K-theory groups. We will see that the Cuntz semiring is given by a ring of (equivalence classes of) certain projections from the enveloping Hopf von Neumann bi-algebra.

Recall that a weight on a C*-algebra is a map $\haar\colon A_{+} \arrow [0,\infty]$ such that $\haar(x+y)=\haar(x)+\haar(y)$ and $\haar(\lambda x)=\lambda\haar(x)$ for all $x$ and $y$ in $A_{+}$ and $\lambda\in[0,\infty).$ Weights appear naturally when we attempt to extend a trace from a C*-algebra to the enveloping von Neumann algebra.
The next Lemma follows from  \cite{ORT} (Proposition 5.2), see also  \cite[]{combes} (Propositions 4.1 and 4.4).
\begin{Lemma}A densely defined semicontinuous tracial weight on a C*-algebra $A$ extends uniquely to a normal tracial weight on the enveloping von Neumann algebra $\Add.$
\label{lem:extending.traces}\end{Lemma}

We recall that an open projection is a~projection in the enveloping von Neumann algebra of a~C*-algebra that is the limit of an increasing net $\{a_\lambda\}$ of positive elements of $A$ in the $\sigma(\Add,A^* )$ topology. There is a bijective correspondence between open projections and hereditary sub-C*-algebras of $A,$ given by $p\mapsto p\Add\! p \cap A.$
The open projection $[a]$ associated with an element $a\in A$ is the strong limit of the sequence $|a|^{1/n}.$ We write $\lim |a|^{1/n} = [a].$
We note that the linear span of the open projections is dense (in norm) in $A.$ This follows, for example, from the spectral theorem. Thus,~elements of $A$ can be approximated by linear combinations of projections of the double dual.

It is well-known that algebra traces correspond to states on the K-theory group. A similar statement, suitably interpreted, is true for Cuntz semigroups.
 We recall that the states on the Cuntz semigroup are the so-called dimension functions. In the case of exact or nuclear C*-algebras, we~can define the dimension functions in a way that will be convenient later on. Given a trace $\haar$ on the algebra, and a positive element $a,$ we may consider using Lemma \ref{lem:extending.traces} to extend the trace to a weight on the double dual, and then applying this weight to the open projection associated with $a$. Denoting this operation by $\widetilde{\haar}([a]),$ the map $a\mapsto \widetilde{\haar}([a])$ is then a dimension function. Since it is necessary to allow the formation of direct sums, we must therefore consider the matrix algebras $A\tensor\M{n}.$ Thus, we~should in fact replace the trace $\haar$ by $\haar\tensor t$ where $t$ is the standard trace on $\M{n}.$ We may use the compact operators $\compact$ instead of the matrix algebras $\M{n},$ where $ n=1,2,3,\cdots,$ and if this is done, then we~tensor with the standard trace on $\compact$ when constructing dimension functions.

Following the discussion in  \cite{{DawsLePham}}, given a C*-algebraic quantum group, $A$, with coproduct $\coproduct{}\colon A\arrow A\tensor A,$ we can consider $A\subset \BH$ in the universal representation, so that both $A\tensor A$ and $\Add\tensor \Add$ are subalgebras of $\BHH.$ We can thus extend the coproduct homomorphism $\coproduct{}\colon \Add\arrow \Add\tensor \Add,$ obtaining, as is well-known, the enveloping Hopf (or Kac)--von Neumann algebra $(\Add,\coproduct{})$ associated with $A.$
\begin{Proposition} In a Kac--von Neumann bi-algebra, if $p=\lim_{n\arrow\infty} a^{1/n}$ and $q=\lim b^{1/n}$ are open projections, and $a$ and $b$ are positive elements of the underlying compact-type C*-algebraic quantum group, the strong limit of $a^{1/n} \convolve b^{1/n}$ is $p\convolve q.$ \label{prop:limits.of.box.and.convolve}
\end{Proposition}
\begin{proof} The convolution product of positive operators is a positive operator (as follows from, e.g.,  \cite[]{EnockSchwartz} (Theorem 1.3.3.i)).
Since $a^{1/(n+1)} -a^{1/n}$ is a positive operator, and similarly  $b^{1/(n+1)} -b^{1/n}$ is positive, we have
$a^{1/n} \convolve b^{1/n} \leq  a^{1/(n+1)}  \convolve b^{1/n}\leq a^{1/(n+1)} \convolve b^{1/(n+1)}. $
 Since $p -a^{1/n}$ and $q -b^{1/n}$ are positive operators, we similarly have
$a^{1/n} \convolve b^{1/n} \leq  p \convolve q. $
This shows that the sequence of positive operators $a^{1/n} \convolve b^{1/n}$ increases and is bounded above, and therefore converges (strongly) in the von Neumann algebra (for example, by  \cite[]{pedersen.book} (Lemma 2.4.4)).

That the limit is $p\convolve q$ remains to be proven. 
 Using an identity from  \cite[]{kahng} (Proposition 3.11), we have
$a^{1/n} \convolve b^{1/n} = \left(\haar(\antipode\inv\cdot)\tensor\Id\right)\left[\coproduct{}(b^{1/n})(a^{1/n}\tensor\Id) \right],$ and
 $p \convolve q = \left(\haar(\antipode\inv\cdot)\tensor\Id\right)\left[\coproduct{}(q)(p\tensor\Id) \right],$
where $\haar$ is the Haar weight on the Kac--von Neumann algebra (which extends the Haar state on the underlying compact-type C*-algebraic quantum group.)
Since the coproduct homomorphism is normal, we have $\lim_{n\rightarrow\infty} \coproduct{}(b^{1/n})=\coproduct{}(p).$
Since $\coproduct{}(b^{1/n})=\coproduct{}(b)^{1/n}$ is a strongly convergent sequence in $\Add\tensor\Add,$ and since multiplication is jointly continuous, with respect to the strong topology, on norm-bounded subsets, it follows that the sequence $\coproduct{}(b^{1/n})(a^{1/n}\tensor\Id)$  converges in the strong topology.
When we apply the slice map  $\haar(\antipode\inv(\cdot))\tensor\Id$ to the strongly convergent sequence $\coproduct{}(b^{1/n})(a^{1/n}\tensor\Id),$ we obtain the sequence of positive operators $a^{1/n} \convolve b^{1/n},$ which we have shown to converge in the strong topology.
Because the slice map  $\haar(\antipode\inv(\cdot))\tensor\Id$ is closed when viewed as an an operator, it follows that the limit of the sequence
$(\haar(\antipode\inv(\cdot)) \tensor\Id )\left[ \coproduct{}(b^{1/n})(a^{1/n}\tensor\Id)\right]$ is $p\convolve q,$ and therefore that the strong limit of $a^{1/n} \convolve b^{1/n}$ is $p\convolve q.$
  \end{proof}
The above Proposition shows that  the formula  $(\Id\tensor t) V(\ell_1\tensor\ell_2)V\inv=\ell_1 \convolve \ell_2$ given in Proposition \ref{prop:box.and.convolve} is well-behaved under passage to open projections:

\begin{Corollary} Let $A$ be a compact-type C*-algebraic quantum group with tracial Haar state.
Let $p$ and $q$ be open projections of $A.$ Let $t$ be the standard trace on $\compact.$
We have $(\Id\tensor t^{**}) (V(p \tensor q)V\inv)=p \convolve q.$ \label{cor:box.and.convolve2}
\end{Corollary}
 \medskip

Applying the Haar state to both sides of the identity  provided by using the above Corollary, we~have a Corollary that remains valid at the level of Cuntz equivalence classes (because $\haar^{**}\tensor t^{**}$ is an example of a state on the Cuntz semigroup, in the open projection picture, and therefore respects Cuntz equivalence):
\begin{Corollary} Let $A$ be a separable  compact-type  C*-algebraic  quantum group with  tracial Haar state, $\haar.$ Let $p$ and $q$ be open projections of $A.$
We have $(\haar^{**}\tensor t^{**}) (p\boxproduct q)=\haar^{**}(p)\haar^{**}( q),$ where $t$ is the standard trace on~$\compact$. \label{cor:box.and.convolve.normalization}
\end{Corollary}

\begin{Remark} Corollary \ref{cor:box.and.convolve.normalization} together with the already established twisted linear property shows that the product we have constructed has the properties required in Definition 2.1 of  \cite{K.nontracial}.  \end{Remark}

The dual object of a compact-type C*-algebraic quantum group $A$ is, in the setting provided by~\cite{BS}, a discrete-type C*-algebraic quantum group. 
From the discussion at the start of Section \ref{Section4}, recall that there is a Fourier transform on $A,$  defined by
$\beta(a,\F(b))=\haar(ab),$ 
and a convolution product, $\convolve,$ in effect defined by $\F(a\convolve b)=\F(a)\F(b).$ We remark that the Fourier transform is invertible on its range, and that
\begin{equation*}\beta(a,w)=\haar(a\F\inv(w))\label{eq:inv.pairing}\end{equation*} for all $w$ in the range of the Fourier transform.
The next lemma  gives some properties of C*-algebraic isomorphisms that intertwine Haar states: the continuity comes from the open mapping theorem. 

\begin{Lemma} Let $A, B$ be compact-type Hopf $C^*$-algebras with faithful  Haar states. Let $f:A\arrow B$ be a~$C^*$-algebraic isomorphism that intertwines the left Haar states. Then, we have the identity
 $\F_B \compose f = f^*\inv\compose\F_A,$ where $\F$ denotes the Fourier transform, and $f^*\inv$  is the inverse map for $f^*,$ the map induced by $f$ on the dual algebras. The isomorphism $f$ takes the domain of $\F_A$ to the domain of $\F_B.$ The linear map $f^*\inv$ is continuous and takes the range of $\F_A$ to the range of $\F_B.$
\label{lem:pullbacks.and.Fourier}\end{Lemma}

 We now consider the map induced on the dual by a  $\fancyK$-co-multiplicative map.
\comment{An unbounded example of a co-tracial linear functional is the trace on the dual.
I think that this trace is in some cases the co-unit composed with the inverse Fourier transform, but since we may not have co-unit projs we have to take limits of co-tracial functionals of the form $\haar(e_n\cdot)$ where $(e_n)$ is an approximate unit for the co-commutative subalgebra.}
\begin{Lemma} Let $A$ and $B$ be compact-type Hopf C*-algebraic quantum groups that have a tracial Haar state, and are separable and nuclear at the C*-algebraic level.
Let $f:A\longrightarrow B$ be a C*-isomorphism that is  $\fancyK$-co-multiplicative and intertwines Haar states. \label{lem:cotraces}\label{lem:de-projectivization}
The map $f^*$ induced by $f$ on the dual algebra(s) satisfies $\haarh(f^*\inv(y_1 y_2))= \haarh(f^*\inv(y_1) f^*\inv (y_2)),$ for all finitely supported $y_i\in\Ah,$  where $\haarh$ is the Haar state of $\Bh.$
\end{Lemma}

\begin{proof}
A C*-morphism $f\colon A\arrow B,$ with $B$ non-degenerately represented on a Hilbert space can be extended to $f\colon\Add\arrow B''$ where $B''$ is the von Neumann algebra generated by $B$ in that representation. In the case that the representation of $B$ is the universal representation,  $B''=B^{**}.$
Thus, we may extend the C*-isomorphism  $f:A\longrightarrow B$ to  $f:\Add\longrightarrow \Bdd,$ where $\Add$ and $\Bdd$ are the enveloping Kac--von Neumann algebras of $A$ and $B,$ respectively (in the universal representation).

A C*-isomorphism necessarily takes a hereditary subalgebra to a hereditary subalgebra, and therefore the extended map $f$ takes an open projection to an open projection.
From the hypothesis that the given map $f$ respects the product $\boxproduct$ on Cuntz equivalence classes, it follows that for open projections $p$ and $q,$  the two open projections $f(p)\boxproduct f(q)$ and $ f(p\boxproduct q)$ are equivalent.  In the open projection picture of the Cuntz semigroup, equivalence of open projections means precise 
 equality under certain tracial weights, as discussed earlier. Thus, $$ (T^{**}\tensor t^{**})(f(p)\boxproduct f(q))=(T^{**}\tensor t^{**})(f(p\boxproduct q)),$$
where $t$ is the standard  trace on $\compact,$  $T^{**}$ is the normal extension of a trace $T$ on  $B$ to the double dual  (see~Lemma \ref{lem:extending.traces}), and $p$ and $q$ are open projections of $\Add.$

By Corollary \ref{cor:box.and.convolve2}, the above expression reduces to
$ T^{**}(f(p)\convolve f(q))=T^{**}(f(p\convolve q)).$

Noting that the above expression is (bi)linear in $p$ and $q$, we may replace  $p$ and $q$  by  finite linear combinations of open projections. By the spectral theorem, a positive operator  can be written as the strong limit of an increasing sequence of elements $p_n$ that are finite linear combinations of open projections. Since $f$ and $T^{**}$ are normal maps, and since the convolution of positive operators is positive, replacing $p$ in the above by $p_n$ and taking the limit, we obtain the case where $p$ is a positive operator. Doing the same with the other variables, we thus have
 $ T^{**}(f(x)\convolve f(y))=T^{**}(f(x\convolve y)),$
where $x$ and $y$ are positive elements of $A.$ We next suppose that $x$ and $y$ are algebraic elements.

Since $f$ intertwines Haar states, Lemma \ref{lem:pullbacks.and.Fourier} shows that  $\F_B \compose f= f^*\inv\compose\F_A .$ Inserting this and the definition of convolution into the above, we conclude that, denoting the Fourier transforms of $x$ and $y$ by $\hat{x}$ and $\hat{y}$ respectively,
\begin{equation*} g(f^*\inv ( \hat{x} ) f^*\inv (\hat{y}))=g(f^*\inv(\hat{x} \hat{y})),
\label{eq:bilin1}\end{equation*}
where $g$ is a linear functional of the form $T^{**}\circ\F\inv_{B}.$ We now make a specific choice of trace $T^{**}.$ Since~the co-unit map of $B$ is a *-homomorphism, it extends by Proposition \ref{lem:extending.traces} to a normal trace on $\Bdd,$ and we use this as our choice of trace.
A calculation with Fourier transforms, following for example  Proposition 4.8 in \cite{VanDaele1998}, gives $\counit_B (a)=\haarh\compose \F_B (a)$ for all $a$ in the domain of the Fourier transform, where $\counit_B$ denotes the co-unit homomorphism of $B,$ $\F_B$ is the Fourier transform defined by the Haar state of $B,$ and $\haarh$ is the  Haar weight of the dual algebra, $\Bh.$
We thus obtain
\begin{equation*} \haarh(f^*\inv ( \hat{x} ) f^*\inv (\hat{y}))=\haarh(f^*\inv(\hat{x} \hat{y})),
\end{equation*}
where $\haarh$ is the Haar weight of $\Bh.$

We had supposed that $\hat{x}$ and $\hat{y}$ are the Fourier transforms of some algebraic elements, $x$ and $y,$ assumed to be positive at the C*-algebraic level.  By passing to finite sums, we can drop the condition that the elements $x$ and $y$ are positive in the C*-algebra, since the usual C*-algebraic decomposition of a element into four positive elements can be made to work in any unital *-closed algebra, and the algebraic elements are a unital *-closed algebra (see, for example, \cite[]{timmermann} (Theorem 5.4.1)).
We thus see~that
\begin{equation}\haarh(f^*\inv(y_1 y_2))= \haarh(f^*\inv ( y_1 ) f^*\inv (y_2)),
\label{eq:bilin2}\end{equation}
where $y_1,y_2\in\Ah$ are Fourier transforms of algebraic elements of $A.$ We now show that  $y_1,y_2\in\Ah$ are themselves algebraic elements, and that the algebraic elements of the dual are  exactly the compactly supported elements of $\Ah.$ The Fourier transform takes the algebraic elements onto the algebraic elements---see for example \cite[]{kahng} (Proposition 3.5 and Theorem 3.8)---and  thus what remains to be shown is that the algebraic elements of $\Ah$ are, in C*-algebraic terminology, the compactly supported elements of $\Ah.$
The algebraic elements $A_0$ in $A$ are a dense compact-type algebraic quantum group. The dual of  $A_0$ is a multiplier Hopf algebra. Thus, the dual of $A_0$ is an  algebraic direct sum inside $\Ah.$ At the level of algebras, $\Ah$ is a $c_0$-direct sum of matrix algebras, and so the dual of $A_0$ is, as claimed, exactly the  subalgebra of compactly supported elements inside $\Ah.$ It follows that the $y_i$ in the above expression \eqref{eq:bilin2} are compactly supported elements, as was to be shown.
 \end{proof}

\begin{Lemma}[\protect{\cite[]{kuce.classify.Hopf.by.Ktheory} (Lemma 2.9)}]  Let $A$ and $B$ be  compact-type  C*-algebraic quantum groups. Let $\alpha\colon A\arrow B$ be a  *-isomorphism, and let ${\alpha^*}\colon \Bh\arrow\Ah$ be its induced action on the dual. We suppose that the action $\alpha^*$ on the dual is a Jordan *-isomorphism. Then, either $\alpha^*$ is multiplicative, or $\alpha^*$ is anti-multiplicative. \label{lem:jordan.auto} \end{Lemma}

 \begin{Proposition}[\protect{\cite[]{kuce.classify.Hopf.by.Ktheory} (Proposition 3.1)}] Let $A$ and $B$ be compact-type C*-algebraic quantum groups, with  tracial Haar states. Let $f\colon A\longrightarrow B$ be an algebra map, intertwining co-units. Let the induced map $f^*\colon \Bh\longrightarrow\Ah$ on the duals be Jordan. Then, $f$ intertwines antipodes.\label{prop:Hopf.Jordan.maps.preserve.antipodes}
\end{Proposition}

We now give a slight modification of the argument from Section 6 of  \cite{kuce.classify.Hopf.by.Ktheory}.
\begin{Theorem} Let $A$ and $B$ be compact-type C*-algebraic quantum groups with  tracial Haar states. Let $f\colon A\longrightarrow B$ be
a C*-isomorphism that intertwines Haar states and co-units. We suppose that the induced map on Cuntz semigroups  intertwines the  products $\boxproduct_A$ and $\boxproduct_B$. Then, $A$ and $B$ are isomorphic or co-anti-isomorphic as Hopf~algebras.\label{th:banacheweski.Hopf}
\end{Theorem}
 \begin{proof}  Note that $f^*$ intertwines the Haar states of $\Ah$ and $\Bh,$ due to the fact that $f$ intertwines the co-units of $A$ and $B.$ (See \cite[]{VanDaele1998} (Proposition 4.8)). Thus, we may regard $f^{*}$ as a linear map that intertwines the trace on certain representations of $\Ah$ and~$\Bh.$ {(Since the Haar states of $\Ah$ and $\Bh$ are tracial}, {they are in fact determined by the size of each block in the $c_0$-direct sum decomposition of $\Ah$ and $\Bh.$ Thus, the Haar state is actually encoded at the C*-algebraic level. See} \cite{fima}(Proposition 2.1) and the discussion there.)

Let $y_1$ and  $y_2$ be  compactly supported elements of $\Ah.$
 Lemma \ref{lem:cotraces}  shows
$\haarh_{\Bh}(f^*\inv(y_1 y_2))= \haarh_{\Bh}(f^*\inv(y_1) f^*\inv (y_2))$
where the map $f^*\inv\colon\Ah\arrow\Bh$ is the inverse of the pullback of the given map. The same proof holds for products of finitely many $y_i,$ so that
$\haarh_{\Bh} ({f\inv}^{*}(y)^{n})=\haarh_{\Bh} ({f\inv}^{*}(y^n))=\haarh_{\Ah} (y^n),$
and $\haarh_{\Ah}$ (resp. $\haarh_{\Bh}$) is the tracial Haar state of the discrete C*-algebraic quantum group $\Ah$ (resp. $\Bh$.)
Thus, we have not only that ${f\inv}^*(y)$ and $y$ have the same values under the trace on $\Bh\subset \BH$ and $\Ah\subset\BH,$ but that this is true when $y$ is replaced by $y^n.$
A linear map of matrix algebras that preserves the trace of every power necessarily preserves the spectrum of elements.  We conclude that $f^*\inv\colon\Ah\arrow\Bh$ preserves the spectrum of compactly supported elements.

{Approximating a general element of the $c_0$-direct sum of matrix algebras $\Ah$ by compactly supported elements, we see that the map $f^*\inv\colon\Ah\arrow\Bh$ preserves the spectrum of operators in general.}
A bijective spectrum-preserving map of  C*-algebras that are $c_0$-direct sums of matrix algebras is  a~Jordan map~\cite[]{aupetit3} (Theorem 3.7). By Proposition  \ref{prop:Hopf.Jordan.maps.preserve.antipodes}, we then furthermore have that $f$ intertwines antipodes.
This implies (by, for example,  \cite[]{kahng} (Lemma 3.3)) that the pullback $f^*$ preserves the C*-involution. By Lemma \ref{lem:jordan.auto}, the pullback map $f^*$ is now either multiplicative or anti-multiplicative. We~thus have, by duality, that $f$ is either an isomorphism or an anti-isomorphism of bi-algebras.
 It follows from uniqueness of the Hopf algebra antipode(s) that $f$ is a Hopf algebra (co-anti)isomorphism.
 \end{proof}
If, in the conclusion of the above Theorem, the given map $f$ is in fact a  co-anti-isomorphism, it~then follows that it reverses the product $\boxproduct.$
By hypothesis, $f$ intertwines the products $\boxproduct_A$ and $\boxproduct_B.$ Therefore, either
f cannot be a co-anti-isomorphism or the Cuntz semirings must be commutative.
We~thus have a remark:
\begin{Remark}If, in the above Theorem, the Cuntz semirings are not commutative, then in fact we obtain a Hopf algebra isomorphism of $A$ and $B.$\end{Remark}
We also mention that the above Theorem can be interpreted as an obstacle to deformation of the~co-product.
\begin{Remark} Suppose that on a given unital C*-algebra, we have a family of co-products, parametrized by $\hbar,$ and suppose that it happens that the Cuntz semiring is invariant with respect to $\hbar.$ Suppose  that the co-unit and tracial Haar state are also invariant with respect to $\hbar.$ By the above Theorem, all of these C*-algebraic quantum groups are bi-isomorphic.\end{Remark}
The above remark thus says that within certain classes of C*-algebraic quantum groups, algebras that are nearby in an appropriate sense are bi-isomorphic. A slightly similar result, for the case of group algebras and using a different notion of closeness, is in \cite{NgCloseness}.

\section{Isomorphism Results and Remarks on K-Theory}\label{Section5}
It is interesting to be able to lift isomorphisms of Cuntz semirings to isomorphisms of C*-algebraic quantum groups. In order to deduce such results from the previous section, we need to recall some of the conclusions of the classification program for C*-algebras.
In some cases, Cuntz semigroup maps can be lifted to algebra maps. A class of C*-algebras  within which this is true will be called a~\textit{classifiable class.} An example of a class of  nonsimple C*-algebras that is classifiable by Cuntz semigroups is the class of separable C*-algebras that are inductive limits of continuous-trace C*-algebras with one-dimensional tree spectrum given in  \cite{CESclassify}. We call this the CES 
class.
 By applying Theorem \ref{th:banacheweski.Hopf} to the C*-algebraic map provided by \cite{CESclassify} we have the following Corollary:
\begin{Corollary} Let $A$ and $B$ be compact-type C*-algebraic quantum groups, belonging to the CES class, or more generally some classifiable class, and having  tracial Haar states. Let $f\colon \Cu(A)\longrightarrow \Cu(B)$ be
an isomorphism of Cuntz semigroups that intertwines the dimension functions coming from the Haar states and the co-units. We suppose that the induced map on Cuntz semigroups  intertwines the  products $\boxproduct_A$ and $\boxproduct_B$. Then, $A$ and $B$ are isomorphic or co-anti-isomorphic as Hopf algebras.\label{cor:Cuntz.banacheweski.Hopf}
\end{Corollary}
 We mention that the above result applies, for example, to a subclass of AT 
 algebras. AT algebras are  inductive limits of direct sums of matrix algebras over the group $S^1.$ 
The above result can be viewed as related in spirit to  Wendel's classic theorem \cite{wendel1952} on lifting group algebra isomorphisms to isomorphisms of groups: we have replaced the group algebra by a much more abstract object. However, perhaps the real significance of results such as the above is that they tell us that in many cases a compact-type Kac C*-algebraic quantum group is completely determined by a certain (semi)ring that is much smaller than the quantum group, together with a knowlege of the C*-algebraic class that the quantum group belongs to. The point is that a much smaller object summarizes most of the information in the quantum group.

Since the K-theory group is generally a smaller object than the Cuntz semigroup, we now briefly consider  K-theory. In our earlier results \cite{kuc.autos,kuce.classify.Hopf.by.Ktheory0,kuce.classify.Hopf.by.Ktheory}, the idea was to lift maps from the K-theory ring of a~finite or discrete C*-algebraic quantum group to bialgebra (co-anti)isomorphisms. For this to work, the~K-theory should form a ring
 with respect to a convolution operation, as in our Proposition \ref{prop:box.and.convolve}.

The methods we used previously depended on discreteness. Thus, the difficulty with a sweeping generalization of our earlier results is:  it is not clear that the K-theory group of a compact-type C*-algebraic quantum group does form a (semi)ring with respect to a convolution product operation.                                   {(Note that, in the algebraic case, the restriction of rings operation does not respect projective modules, except in special cases. Thus, the algebraic product module $\coproduct{*} (M_1 \otimes M_2), $ need not be a projective module even if $M_1$ and $M_2$ are projective.)} 
We address this issue next.

Clearly,  we can consider the equivalence classes of the projections within the Cuntz semigroup. Moreover,  in the stably finite case, equivalence of projections in the Cuntz semigroup is the same as ordinary equivalence of projections, in matrix algebras over $A$ (see, for example,  \cite[]{KirRor2000} (p. 641)). Let us now make the slightly stronger assumption of stable rank 1.

Recall that, in the unital case, the K-theory group of $A$ can be defined as the enveloping group of the semigroup of equivalence classes of projections  $V(A)$ of $A$. This semigroup $V(A)$ is, in the unital case, generated by the projections of $A,$ and in the stable rank 1 case there is an injection $\imath$ of $V(A)$ into~$W(A).$

Elements of a Cuntz semigroup that are equivalent to a projection are called \textit{projection-class} elements, and elements that are not equivalent to a projection are called \textit{purely positive} elements.
\begin{Proposition}In the tracial and stable rank 1 case, the Cuntz semiring product of two elements is purely positive if and only if one of the elements is purely positive.\end{Proposition}
\begin{proof}In the stable rank 1 case, a positive operator is projection class if and only if either 0 is an isolated point of its spectrum, or 0 is not an~element of its spectrum \cite[]{perera.spectrum} (Proposition 3.12). Thus,~an~element is purely positive if and only if zero is a point of accumulation in the  spectrum of some (any) operator representing it. The spectrum of an operator remains unchanged in a faithful representation (of~a~C*-algebra that it belongs to). Therefore, we can regard the operator $ V(a\tensor b)V\inv$ given by Proposition~\ref{prop:box.and.convolve} as an element of $\BHH$ without altering its spectrum. However, then, $V$ extends to a~unitary, and the spectrum of the operator is equal to the spectrum of $a \tensor b$. The spectrum of $a\otimes b$ is given by the set of all pairwise products $\{\lambda\mu\, |\, \lambda\in \mbox{Sp}(a),  \mu\in \mbox{Sp}(b) \}.$ Since the spectrum of $b$ is never empty, it follows that if the spectrum of $a$ accumulates at zero, so does the spectrum of $a\tensor b$. Similarly, if the spectrum of $b$ accumulates at zero, so does the spectrum of $a\tensor b$.  The result follows.\end{proof}

 Thus, the product $\boxproduct$ gives a product on  the (image of) the semigroup of equivalence classes of projections, $\imath(V(A)).$ This gives a product on the semigroup $V(A),$
but products on semigroups do not always pass to products on the enveloping group:  the precise condition that is needed for products to extend in general is that the semigroup must be cancellative \cite{Dale}, and C*-algebras having the property that their semigroup of projections is cancellative are said to have cancellation. It is known that stable rank one implies cancellation  \cite[]{blackadar} (Proposition 6.5.1). It then follows  that:
\begin{Corollary} If a unital and separable C*-algebraic quantum group $A$ has a faithful tracial Haar state and has stable rank 1, then the product $\boxproduct$ gives a product on  $K_0 (A)$. \label{cor:KtheoryProducts}
\end{Corollary}

We now give an example of computing the product on K-theory in terms of the product on the projection monoid:
\begin{Example}  Recall that elements of $K_0$ are by definition formal differences of projections. The canonical product on K-theory is
$([p_1]-[q_1])\boxproduct([p_2]-[q_2]):=[p_1\boxproduct p_2]+[q_1\boxproduct q_2]-[p_1\boxproduct q_2]-[q_1\boxproduct p_2].$

It is known that we can use the equivalence relations on K-theory to write any element in the form $[p]-[I_n]$ where $I_n$ is the identity element of $M_n(A)$.  The decomposition is unique, up to equivalence of $p$, if we take $n$ to be minimal. Recall that if $\Id$ is the multiplicative identity of $A,$ then
$[\Id]\boxproduct [a] $ is equivalent to $\Id$ for all $a$~(\textit{c.f.}~Remark \ref{rem:multiplication.table}.)  Now the above product simplifies to
\begin{equation}([p_1]-[I_n])\boxproduct([p_2]-[I_m])=[p_1\boxproduct p_2]-[I_k].
\label{eq:naive.product}\end{equation}\end{Example}
We thus see that that the canonical product on K-theory coincides, in this case, with the ``naive'' product defined by the right hand side of Equation \eqref{eq:naive.product}.

Choosing a K-theoretically classifiable class of C*-algebras, we  expect to be able to lift K-theory ring maps to bi-algebra (anti)isomorphisms. The class of nonsimple Approximately Interval algebras, usually called AI algebras, 
satisfying a mild condition called the ideal property, are classifiable by K-theory and traces, see for example  \cite{JJ,Stevens}, and non-simple AI algebras that are of real rank zero are classified, in  \cite{ElliotAI}, by their $K_0$ and $K_1$ groups. The rather large class of approximately subhomogeneous (ASH) real rank zero C*-algebras is classified by their K-theory groups in  \cite{DadarlatGong}.

We mention, for example, a variant of Theorem \ref{th:banacheweski.Hopf}.
\begin{Theorem} Let $A$ and $B$ be compact-type C*-algebraic quantum groups with  tracial Haar states. Assume that $A$ and $B$ have stable rank one and real rank zero at the C*-algebra level. Let $f\colon A\longrightarrow B$ be
a C*-isomorphism that intertwines Haar states and co-units. We suppose that the induced map on the $K_0$-groups  intertwines the  products $\boxproduct_A$ and $\boxproduct_B$. Then, $A$ and $B$ are isomorphic or co-anti-isomorphic as C*-algebraic quantum groups.\label{th:banacheweski.Hopf.sr1}
\end{Theorem}
The proof is similar to the proof of Theorem \ref{th:banacheweski.Hopf}, but simpler, because the real rank condition allows us to work with linear combinations of projection elements from the algebra and K-theory states on~them, rather than open projections  from the double dual.

In order to replace C*-isomorphisms by K-theory isomorphisms in Theorem \ref{th:banacheweski.Hopf.sr1}, we must choose a~class of C*-algebra that is classifiable by K-theory. We can, for example,  choose ASH algebras:   let $K_*$ denote the direct sum of K-theory groups used in  \cite{DadarlatGong},  with a  product on $K_0$  as defined by Equation \eqref{eq:naive.product}. The K-theory state associated with the tracial Haar state is the map induced on $K_0$ by the Haar state.

\begin{Theorem}Let $A$ and $B$ be compact-type C*-algebraic quantum groups with  faithful tracial Haar states. Assume that $A$ and $B$ are slow dimension growth ASH algebras with real rank zero  at the C*-algebra level. Let $f\colon K_{*}(A) \longrightarrow K_{*}(B)$ be
an isomorphism that intertwines the K-theory state associated with the tracial Haar states. We suppose that the map on the $K_0$-groups  intertwines the  products $\boxproduct_A$ and $\boxproduct_B$. Then, $A$ and $B$ are isomorphic or co-anti-isomorphic as C*-algebraic quantum groups.\label{th:banacheweski.Hopf.ASH}  \end{Theorem}

The condition of stable rank 1 has been dropped in the above because slow dimension growth implies \cite[]{DadarlatGong} (Propositions 3.9 and 2.3) cancellation of projections, which is sufficient.
We note that the above result applies to a subclass of AT algebras.

\subsection*{Acknowledgements} {I thank the referee for comments that improved this manuscript and for shortening the proof of  Lemma \ref{lem:special.module}. I thank my teachers and mentors, George Elliott, Gert Pedersen, and Georges Skandalis. {Also thanks to NSERC for financial support.}



\end{document}